\documentclass[reqno,12pt]{amsart}

\usepackage{geometry,amsmath,amsfonts,amssymb,amsthm,amscd,mathrsfs,graphicx,enumerate,multicol,wrapfig,subfigure,hyperref}
\usepackage[all]{xy}
\usepackage[square, numbers, comma, sort&compress]{natbib}
\usepackage{setspace}

\usepackage[T1]{fontenc}

\numberwithin{equation}{subsection} 

\newcommand{\ra}{\rightarrow}

\newcommand{\lra}{\longrightarrow}

\newcommand{\p}{\prime}
\newcommand{\pt}{\partial}

\newcommand{\al}{\alpha}
\newcommand{\Om}{\Omega}
\newcommand{\om}{\omega}
\newcommand{\gam}{\gamma}

\newcommand{\s}{\sigma}
\newcommand{\vp}{\varphi}
\newcommand{\lam}{\lambda}
\newcommand{\Lam}{\Lambda}
\newcommand{\q}{\theta}

\newcommand{\dt}{\delta}

\newcommand{\Zbb}{\mathbb{Z}}

\newcommand{\Cbb}{\mathbb{C}}

\theoremstyle{plain} 

\newtheorem{THM}{Theorem}[section]
\newtheorem{DEF}[THM]{Definition}

\newtheorem{CON}[THM]{Construction}

\newtheorem{QUE}[THM]{Question}
\newtheorem{ILL}[THM]{Illustration}
\newtheorem{PROP}[THM]{Proposition}
\newtheorem{LEM}[THM]{Lemma}
\newtheorem{COR}[THM]{Corollary}
\newtheorem{REM}[THM]{Remark}

\newtheorem*{THMS}{Theorem}

\newcommand{\bt}{\bullet}

\newcommand{\Ac}{\mathcal{A}}
\newcommand{\Bc}{\mathcal{B}}

\newcommand{\img}{\mathrm{im}}

\newcommand{\Vc}{\mathcal{V}}
\newcommand{\Uc}{\mathcal{U}}
\newcommand{\Wc}{\mathcal{W}}
\newcommand{\Oc}{\mathcal{O}}
\newcommand{\red}{\mathrm{red}}

\newcommand{\Cc}{\mathcal{C}}

\newcommand{\Fc}{\mathcal{F}}

\newcommand{\Xfr}{\mathfrak{X}}
\newcommand{\Ec}{\mathcal{E}}
\newcommand{\Tc}{\mathcal{T}}

\newcommand{\Ufr}{\mathfrak{U}}

\newcommand{\Mfr}{\mathfrak{M}}
\newcommand{\Dfr}{\mathfrak{D}}
\newcommand{\Scl}{\mathcal{S}}
\newcommand{\Xc}{\mathcal{X}}

\usepackage{xcolor}
\definecolor{airforceblue}{rgb}{0.36, 0.54, 0.66}
\definecolor{burgundy}{rgb}{0.5, 0.0, 0.13}
\definecolor{majorelleblue}{rgb}{0.38, 0.31, 0.86}
\definecolor{darkblue}{rgb}{0.0, 0.0, 0.55}

\hypersetup{urlcolor=burgundy, linkcolor=burgundy,
citecolor=darkblue,
colorlinks=true}

\title{
On The Problem of Splitting Deformations of Super Riemann Surfaces
\\}
\author{\small Kowshik Bettadapura}
\date{}

\begin{document}
\maketitle

\begin{abstract}
\noindent 
An odd deformation of a super Riemann surface $\Scl$ is a deformation of $\Scl$ by variables of odd parity. In this article we study the obstruction theory of these odd deformations $\Xc$ of $\Scl$. We view $\Xc$ here as a complex supermanifold in its own right. Our objective in this article is to show, when $\Xc$ is a deformation of second order of $\Scl$ with genus $g>1$: if the primary obstruction class to splitting $\Xc$ vanishes, then $\Xc$ is in fact split. This result leads naturally to a conjectural characterisation of odd deformations of $\Scl$ of any order.
\end{abstract}

\setcounter{tocdepth}{1}
\tableofcontents

\onehalfspacing

\section{Introduction}

\noindent
Super Riemann surfaces were first introduced by Friedan in \cite{FRIEDAN} in the context of superstring theory. Following this, a flurry of activity took place with the aim of establishing super Riemann surfaces as objects of independent, mathematical interest. Some relevant articles are \cite{RABCRANE, RABTORI, FALQMOD}. In these cited works one will find, for super Riemann surfaces: Teichm\"uller theory, the uniformization principle, projective embeddings; and a study of the moduli stack of super Riemann surfaces $\Mfr_g$. It is $\Mfr_g$, which has come to find a resurgence of interest. 
\\\\
A recent result of Donagi and Witten in \cite{DW1, DW2} asserts: $\Mfr_g$, as an object in supergeometry,\footnote{in \cite{WITTRS, DW1} it is mentioned that $\Mfr_g$ should be thought of as the analogue in supergeometry to an orbifold or stack; or as a supermanifold twisted by a $\Zbb_2$-gerbe.} is non-split in genus $g\geq 5$.\footnote{In fact a stronger result is obtained: that a holomorphic projection of $\Mfr_g$ to its reduced space is obstructed. This implies $\Mfr_g$ is itself non-split.} Such a result was suspected by the theoretical physics community for some time. Indeed, in \cite{FALQMOD} a heuristic argument is provided for the non-splitness of $\Mfr_g$ in genus $g\geq 3$. Now while $\Mfr_g$ motivates our considerations in this article, we refrain from making any statements about $\Mfr_g$ itself. Our focus is on classifying the objects it parametrises: super Riemann surfaces. 
\\\\
A deformation of a super Riemann surface $\Scl$ is a complex supermanifold $\Xc$ containing $\Scl$ and equipped with a certain atlas reflecting the structure of $\Scl$, which we call a \emph{superconformal atlas}. This viewpoint is inspired by that of Kodaira and Spencer in their construction of deformations of compact, complex manifolds in \cite{KS}. We are guided throughout this article by the desire to understand the relationship between: 
\begin{enumerate}[(i)]
	\item the obstruction theory for deformations of $\Scl$ as a supermanifold and;
	\item the deformation theory of $\Scl$. 
\end{enumerate}
Obstruction theory for (complex) supermanifolds is motivated by the effort to classify supermanifolds into two classes: \emph{split} and \emph{non-split}. In the early works of Berezin, collected in \cite{BER}; Green in \cite{GREEN} and Manin in \cite{YMAN} a general framework is established to address this classification problem. However, it is nevertheless quite difficult to conclude when a given supermanifold will be split. 
\\\\
In this article we will deduce necessary and sufficient conditions for an odd, second order deformation $\Xc$ to be split. That is, we prove (Theorem \ref{tgguygvg3fg83h3}):

\begin{THMS}
An odd, second order deformation $\Xc$ of a super Riemann surface of genus $g>1$ is split if and only if its primary obstruction class vanishes.
\end{THMS}

To emphasize: supermanifolds of dimension $(p|2)$ are classified by their primary obstruction class. This is generally false in dimension $(p|q)$ with $q>2$ however, and odd deformations $\Xc$ of second order have dimension $(1|3)$. Hence the above result is perhaps a little surprising. As it results directly from the superconformal structure on $\Xc$ it is natural to question whether the above theorem will hold in more generality (Question \ref{fh39hf9839323}). 
\\\\
Given that $\Mfr_g$ is generally non-split, a greater understanding of split and non-split structures in complex supergeometry is desirable.

\subsection{Outline}
This article is divided into three sections (excluding this introduction and concluding remarks). In what follows we briefly summarise of each section. 

Section \ref{prelim} consists of preliminary theory concerning supermanifolds and super Riemann surfaces. In Section \ref{nfu3h9j39j39dj3kd02} we consider odd deformations of super Riemann surfaces. A study of them and their Kodaira-Spencer map is presented with a focus on odd deformations of second order. We conclude with a construction of a non-split deformation with vanishing Kodaira-Spencer map. In Section \ref{r83jf893jf0j30} we look at odd, second order deformations as complex supermanifolds. Our main result on splitting these deformations is Theorem \ref{tgguygvg3fg83h3} whose proof occupies much of this section.

\subsection*{Acknowledgements}
The author would like to firstly acknowledge the Australian National University where much of this work was undertaken; the anonymous referee for useful suggests on the improvement of this article; and finally the Yau Mathematical Sciences Centre, Tsinghua University, where support was given while preparing revisions of this article.

\section{Preliminaries}
\label{prelim}

\subsection{Complex Supermanifolds}
Many the concepts from geometry, both algebraic and differential, generalise to the setting of supergeometry. Standard references include \cite{LEI, BER,YMAN, QFAS}. We will briefly summarise here some relevant notions for our purposes in this article. We begin with that of a supermanifold: 
\\\\
Fix a rank $q$, holomorphic vector bundle $E$ over a $p$-dimensional, complex manifold $M$. Denote by $\Ec$ the sheaf of holomorphic sections of $E$. With this data we define a complex supermanifold as follows:

\begin{DEF}
\emph{A $(p|q)$-dimensional, complex supermanifold modelled on $(M, E)$ is a locally ringed space $\Xfr_{(M,E)} = (M, \Oc_M)$ where $\Oc_M$ is globally $\Zbb_2$-graded and locally isomorphic to the sheaf of algebras $\wedge^\bt\Ec$. The complex manifold $M$ is called the \emph{reduced space} of $\Xfr$ and the bundle $E$ is called the \emph{modelling bundle} of $\Xfr$}
\end{DEF}

The space $\Pi E := (M, \wedge^\bt\Ec)$ will be a $(p|q)$-dimensional supermanifold, albeit a particularly trivial one. It is referred to as the \emph{split model}.  Any supermanifold $\Xfr_{(M,E)}$ modelled on a pair $(M, E)$ will be locally isomorphic to $\Pi E$. This fact serves as a basis on which to start a program of classification.

\begin{DEF}\label{4j9fj409gjo3jg30}
\emph{A supermanifold $\Xfr_{(M, E)}$ is said to be \emph{split} if it is isomorphic to $\Pi E$. Otherwise, it is said to be \emph{non-split}. An isomorphism $\Xfr_{(M, E)}\stackrel{\cong}{\ra}\Pi E$, if one exists, is referred to as a \emph{splitting map}.
}
\end{DEF}

\subsection{An Atlas for Supermanifolds}
Just like manifolds, supermanifolds admit an atlas consisting of a collection of open sets $\Ufr = \{\Uc, \Vc, \ldots\}$, charts $\Uc\stackrel{\cong}{\ra}\Cbb^{p|q}$ and transition data. If $(x^\mu, \q_a)$ (resp. $(y^\mu, \eta_a)$) denote coordinates on $\Uc$ (resp. $\Vc$), then on the intersection $\Uc\cap\Vc$ we can write:
\begin{align}
y^\mu &= \rho_{\Uc\Vc}^\mu(x, \q) = f_{UV}^\mu(x) + \sum_{|I|>0} f_{UV}^{\mu|2I}\q_{2I}
\label{uicuibviuv78v3}
\\
\eta_a &= \rho_{\Uc\Vc, a}(x, \q) = \zeta_{UV, a}^b(x)~\q_b + \sum_{|I|>0} \zeta_{UV, a}^{2I+1}(x)~\q_{2I+1}
\label{porviorhg894h89}
\end{align}
where $I$ is a multi-index and $|I|$ its length; for $I = (i_1, \ldots, i_n)$ that $\q_I = \q_{i_1}\wedge\cdots\wedge\q_{i_n}$; by $2I$ (resp. $2I+1$) it is meant the multi-indices of even, resp. odd, length. If the coefficient functions $\{(f_{UV}^{\mu|2I})\}$, resp. $\{ (\zeta_{UV, a}^{2I+1})\}$, are holomorphic on the intersections $\Uc\cap \Vc$ then we say $\{\rho_{\Uc\Vc}\}$ is holomorphic. The collection $\vartheta = \{\rho_{\Uc\Vc}\}$ are the transition data of a complex supermanifold $\Xfr$. In this article we will term the pair $(\Ufr, \vartheta)$ at \emph{atlas} for $\Xfr$.

\subsection{Splittings and Obstructions}
A method to study the problem of whether a given supermanifold is split was devised by Green in \cite{GREEN} and elaborated on further by Manin in \cite{YMAN} and Onishchik in \cite{ONISHCLASS}. Rather than trying to construct a splitting map, one considers studying instead the obstructions to the existence of such a map. This leads then to a notion of `obstruction theory' for supermanifolds. 
\\\\
In this article we will only be considered with what we term the \emph{primary obstruction class}.\footnote{This class is referred to also as the `first' obstruction class to splitting in works such as \cite{DW1, DW2}.} For a supermanifold $\Xfr$ we can write down its primary obstruction class directly from any atlas. More concretely, we have the following result, a justification for which we leave to \cite[p. 191]{YMAN} and \cite[Ch. 2, 3]{BETTPHD}.

\begin{LEM}\label{pekpfk49fk94p}
The primary obstruction class of $\Xfr_{(M,E)}$, denoted $\om_{(M, E)}$, is a class in $H^1(M, \wedge^2\Ec\otimes\Tc_M)$. With respect to an atlas $(\Ufr, \vartheta)$ for $\Xfr_{(M, E)}$, a cocycle representative $\{\om_{\Uc\Vc}\}$ for $\om_{(M, E)}$ is:
\begin{align}
\om_{\Uc\Vc} :=\sum_{\mu, i , j} f_{UV}^{\mu|ij}~\q_{ij}\frac{\pt}{\pt y^\mu}. 
\label{TVYTRVUYEBIY}
\end{align}
If $\om_{(M, E)}\neq0$, then $\Xfr_{(M, E)}$ is non-split. \qed
\end{LEM}

At this stage a useful observation is the following:

\begin{LEM}\label{ejf9hf98h839f3830}
If $\dim \Xfr_{(M, E)} = (p|1)$, then $\Xfr_{(M, E)}$ is split. If $\dim \Xfr_{(M, E)} = (p|2)$, then $\Xfr$ is split if and only if $\om_{(M, E)}=0$.\qed 
\end{LEM}

In general for a $(p|q)$-dimensional supermanifold $\Xfr_{(M, E)}$ with $q>2$, the vanishing of $\om_{(M, E)}$ implies the existence of higher obstruction classes, living in either $H^1(M, \wedge^{2k}\Ec\otimes \Tc_M)$ or $H^1(M, \wedge^{2k+1}\Ec\otimes \Ec^\vee)$, $k>2$. The supermanifold $\Xfr_{(M, E)}$ will be split if all of these higher obstruction classes to splitting vanish. For a more comprehensive account of (higher) obstruction theory for supermanifolds we refer to \cite{GREEN, ONISHCLASS, DW1, DW2}. A study utilising atlases to study higher obstruction theory is undertaken in \cite{OBSTHICK,BETTPHD}.

\subsection{Super Riemann Surfaces}
The following definition of a superconformal structure on a $(1|1)$-dimensional supermanifold is taken from \cite[p. 23]{DW1}, but can also be found in \cite{FRIEDAN} and in \cite{ROSLYGSCONF} where the term \emph{superconformal manifold} is used.

\begin{DEF}
\label{kclckjnrkcnksee}
\emph{Let $\Scl$ be a $(1|1)$-dimensional supermanifold and suppose $\Dfr\subset \Tc_\Scl$ is a subsheaf of rank $(0|1)$ such that the quotient bundle satisfies, 
\[
\Tc_\Scl/\Dfr \cong \Dfr^{\otimes 2}. 
\]
Then $\Dfr$ is called a \emph{superconformal structure} for $\Scl$. 
}
\end{DEF}

\begin{DEF}\label{rivorhovuh4o84o8j}
\emph{A \emph{super Riemann surface} is a $(1|1)$-dimensional supermanifold $\Scl$ with: (1) reduced space a Riemann surface; and (2) a choice of superconformal structure $\Dfr$. The genus $g$ of $\Scl$ is defined to be that of its reduced space. 
}
\end{DEF}

Not every $(1|1)$-dimensional supermanifold $\Scl$ will admit a superconformal structure. To illustrate we present the following characterisation which appears in \cite{RABCRANE}. 

\begin{THM}
\label{dcnjkfnvhjtbooejie}
Any super Riemann surface $\Scl$ is modelled on the data of a genus $g$ Riemann surface $C$ and a choice of spin structure $\ell$ on $C$, i.e., a line bundle $\ell$ on $C$ equipped with an isomorphism $\ell^{\otimes 2}\stackrel{\cong}{\ra}\Om^1_C$. \qed
\end{THM}
 
For a super Riemann surface $\Scl$, since $\dim \Scl = (1|1)$, we have by Lemma \ref{ejf9hf98h839f3830}:

\begin{COR}
Any super Riemann surface $\Scl$ is split and can be written $\Scl = \Pi \Om^{1/2}_C$ for $\Om^{1/2}_C$ a choice of spin structure on $C$.\qed
\end{COR}

\section{Deformations of Super Riemann Surfaces} 
\label{nfu3h9j39j39dj3kd02}

\subsection{Preliminaries} 
Deformation theory for complex superspaces more generally was established by Vaintrob in \cite{VAIN}.\footnote{However, as pointed out by Donagi and Witten in \cite[pp. 29-30]{DW2}, it is a more subtle issue to construct deformations of non-split superspaces.} There, deformation theory is developed from a  more algebraic viewpoint. Presently, we adopt a more \v Cech-theoretic construction of deformations, motivated by the original work on deformations of compact, complex manifolds by Kodaira and Spencer, as detailed in \cite{KS}. 
\\\\
Let $\vp_\xi: \Cbb^{1|1} \ra\Cbb^{1|1}$ be a morphism with $\xi\in \Cbb^{0|n}$. Write $\vp_\xi = (\vp_\xi^+, \vp_\xi^-)$, where:
\begin{align*}
\vp_\xi^+ &= \lam + \lam^{ij}~\xi_{ij} + \lam^{ijkl}~\xi_{ijkl} + \ldots + f^i~\xi_i\q + f^{ijk}\xi_{ijkl}\q + \ldots
\\
\vp^-_\xi 
&= \zeta~\q + \zeta^{ij}~\xi_{ij}\q + \ldots + \psi^i\xi_i + \psi^{ijk}~\xi_{ijk} + \ldots
\end{align*}
The coefficients $\lam^I, f^I, \zeta^I, \psi^I$ are all functions of $x$. In setting $\xi = (\xi_1, \ldots, \xi_n)$, we have more succinct notation: 
\begin{align}
\vp_\xi^+ = \lam(x, \xi) + f(x, \xi)\q
&&
\mbox{and}
&&
\vp^-_\xi = \zeta(x, \xi)\q + \psi(x, \xi).
\label{387hfh4fhh854jfiorf8}
\end{align}
The following result is well-known and can be found, for instance, in \cite{RABCRANE}.

\begin{LEM}\label{piorjviorjio5of04joijejj0jf039jpiorjviorjio5of04joijejj0jf039j}
The morphism $\vp_\xi$ is superconformal if and only if
\begin{align*}
\zeta^2 = \frac{\pt \lam}{\pt x}  + \psi\frac{\pt\psi}{\pt x}
&&
\mbox{and}
&&
f = \zeta\cdot \psi.
\end{align*}\qed
\end{LEM}

\begin{DEF}
\emph{A family of superconformal morphisms $\vp= \{\vp_\xi\}_{\xi\in \Cbb^{0|n}}$ is said to be a superconformal morphism of \emph{order} $n$}.
\end{DEF}

\begin{DEF}
\emph{Fix a super Riemann surface $\Scl = \Pi \Om^{1/2}_C$. An atlas $(\Ufr, \vartheta)$ for a supermanifold $\Xc$ with $\Xc_\red = C$ is said to be \emph{superconformal} and of order $n$ with respect to $\Scl$ if:
\begin{enumerate}[(i)]
	\item the data $\vartheta$ comprises a collection of families of superconformal morphisms $\{\rho_{\Uc\Vc;\xi}\}_{\Uc,\Vc\in \Ufr}$ of order $n$ and;
	\item	at $\xi = 0$ we have: $\{\rho_{\Uc\Vc;\xi}|_{\xi = 0}\} = \{\rho_{\Uc\Vc}\}$ where $\{\rho_{\Uc\Vc}\}$ are the transition data for the given super Riemann surface $\Scl$.
\end{enumerate}
}
\end{DEF}

\begin{DEF}\label{4hf984hf98f830}
\emph{An odd, $n$-th order deformation of a super Riemann surface $\Scl$ is a supermanifold $\Xc$ which admits a superconformal atlas of order $n$ with respect to $\Scl$.}
\end{DEF}

\begin{REM}
\emph{Odd deformations of order $n = 1$ are referred to as \emph{infinitesimal}. These deformations have been well studied. For $\Scl = \Pi \Om^{1/2}_C$ fixed, the infinitesimal deformations are classified by $H^1(C, \Tc_C^{1/2})$. References for this include \cite{RABCRANE, DW1}. For a detailed treatment we refer also to \cite[Chapter 6]{BETTPHD}. 
}
\end{REM}

\begin{REM}
\emph{In the physics literature, a \emph{deformation} of a super Riemann surface as we have given in Definition \ref{4hf984hf98f830}, is itself referred to as a super Riemann surface. See for instance the definition in \cite[p. 603]{RABCRANE}.}
\end{REM}

For $\Xc$ an odd, $n$-th order deformation of a super Riemann surface $\Scl$ there will exist a morphism of supermanifolds $\pi_\Xc : \Xc\ra \Cbb^{0|n}$. This morphism is given locally by a projection $(x, \q, \xi) \mapsto \xi$ and can be identified with $\Scl$ at $0$, i.e, $\pi_\Xc^{-1}(0) = \Scl\subset \Xc$. Evidently $\Xc$ will be a complex supermanifold of dimension $(1| n+1)$.

\subsection{The Superconformal Atlas $n = 2$}
The goal of the present section is to set-up the preliminary theory with which to study odd deformations of second order in more detail. We begin by taking a closer look at Lemma \ref{piorjviorjio5of04joijejj0jf039jpiorjviorjio5of04joijejj0jf039j}.

\begin{ILL}\label{ijodjriohtsayaubdiebiuh3hjoi}
Suppose $n = 2$ and let $\vp_\xi: \Cbb^{1|1}\ra\Cbb^{1|1}$ be a superconformal morphism with $\xi = (\xi_1,\xi_2)\in \Cbb^{0|2}$. Write $\vp$ as in \eqref{387hfh4fhh854jfiorf8} where,
\begin{align*}
\lam(x, \xi) = \lam^0(x) + \lam^{12}(x)~\xi_{12}
&&
f(x, \xi) = f^1(x) ~\xi_1+ f^{2}(x)~\xi_{2}
\end{align*}
and similarly,
\begin{align*}
\zeta(x, \xi) = \zeta^0(x) + \zeta^{12}(x)~\xi_{12}
&&
\mbox{and}
&&
\psi(x, \xi) = \psi^1(x)~\xi_1 + \psi^2(x)~\xi_2.
\end{align*}
Then from Lemma $\ref{piorjviorjio5of04joijejj0jf039jpiorjviorjio5of04joijejj0jf039j}$, since $\vp_\xi$ is superconformal, we have the relations:
\begin{align}
(\zeta^0)^2 = \frac{\pt\lam^0}{\pt x};
&&
f^i = \zeta^0\psi^i.
\label{opemccmlenrvemlvm}
\end{align}
and 
\begin{align}
\frac{\pt}{\pt x}\lam^{12} = 2\zeta^0\zeta^{12} - \frac{\pt\psi^1}{\pt x}\psi^2 + \psi^1\frac{\pt \psi^2}{\pt x}
\label{opejoivh839v30455j}
\end{align}
Hence any solution to \eqref{opemccmlenrvemlvm} and \eqref{opejoivh839v30455j} will define a superconformal morphism of $\Cbb^{1|1}$ over $\Cbb^{0|2}$. 
\end{ILL}

Now consider a superconformal atlas $(\Ufr, \vartheta)$, where $\vartheta = \{\rho_\xi\}$ is given by \eqref{hbbubeubcdioiorjf4jf984jf904j4fj389f839jf9} and \eqref{dioiorjf4jf984jf904j4fj389f839jf9}. For convenience we recall this below. Firstly:
\begin{align}
 \rho_{\Uc\Vc; \xi}^+ &= f_{UV} + f^i_{UV}~\xi_i\q + g^{12}_{UV}~\xi_{12} ~~
 \mbox{and;}
 \label{y873d73y8dj2d}
 \\
 \rho_{\Uc\Vc,\xi}^- &= \zeta_{UV}~\q +
  \psi^i_{UV}~\xi_i +  \zeta_{UV}^{12}~\xi_{12}\q.
\label{ggfgfGSVVDVDJHB}
 \end{align}
In order for $\rho_\xi = \{\rho_{\Uc\Vc;\xi}\}$ to be superconformal we require the conditions in \eqref{opemccmlenrvemlvm} and \eqref{opejoivh839v30455j} to hold. That is, on the intersections $U\cap V$:
\begin{align}
\frac{\pt f_{UV}}{\pt x} = \zeta_{UV}^2;
&&
f_{UV}^i = \zeta_{UV}\psi_{UV}^i;
\label{ojfoeofhoifjpie}
\end{align}
and
\begin{align}
\frac{\pt g^{12}_{UV}}{\pt x}
=
2\zeta_{UV} \zeta_{UV}^{12} 
-
\left(
\frac{\pt \psi_{UV}^1}{\pt x}\psi_{UV}^2 - \psi_{UV}^1\frac{\pt \psi^2_{UV}}{\pt x}
\right)
\label{648379fh93huf}
\end{align}
We wish to now identify the conditions under which $(\Ufr, \vartheta)$ will patch together to give an atlas for a deformation $\Xc\ra\Cbb^{0|2}$ of $\Scl$. This means the cocycle condition must be satisfied. In even and odd components this condition states that on intersections $\Uc\cap \Vc$:
\begin{align}
\rho^+_{\Vc\Uc}\circ\rho_{\Uc\Vc}
=
\mathrm{id}
&&
\mbox{and}
&&
\rho^-_{\Vc\Uc}\circ \rho_{\Uc\Vc}
=
\mathrm{id};
\label{eopjfiofh8o3hfo3j}
\end{align}
and on triple intersections $\Uc\cap \Vc\cap \Wc$:
\begin{align}
\rho^+_{\Uc\Wc} = \rho^+_{\Vc\Wc}\circ\rho_{\Uc\Vc}
&&
\mbox{and}
&&
\rho^-_{\Uc\Wc} = \rho^-_{\Vc\Wc}\circ \rho_{\Uc\Vc}.
\label{eocneioeoeioe}
\end{align}
We will investigate the cocycle condition on intersections and triple intersections separately in what follows.

\subsubsection{On Intersections}

\begin{LEM}\label{encioeuifhe7f78}
Let $\vartheta = \{\rho_\xi\}$ where $\rho_\xi = \{\rho_{\Uc\Vc;\xi}\}$ is superconformal. If $(\Ufr,\vartheta)$ defines an atlas for a supermanifold $\Xc$, then on all non-empty intersections $U\cap V$,
\begin{align}
g_{UV}^{12} &= -\zeta_{UV}^2g^{12}_{VU}~~\mbox{and;}
\label{keiohuieuksciwni}
\\
\zeta_{UV}^{12}
&=
-
\zeta_{UV}^2~\zeta^{12}_{VU}
-
\frac{\pt \zeta_{UV}}{\pt x}g_{VU}^{12}.
\label{ejioheovhhvo3f332d}
\end{align}
\end{LEM}

\begin{proof}
On a single open set $\Uc$ we have $\rho_{\Uc\Uc}(x, \q) = \mathrm{id}$ which implies:
\begin{align}
f^i_{UU} = g^{12}_{UU} = 0
&&
\mbox{and}
&&
\psi^i_{UU} = \zeta_{UU}^{12}=0.
\label{3hr9737fh398f380}
\end{align}
We will consider firstly the even component in what follows. In imposing \eqref{eopjfiofh8o3hfo3j} and using \eqref{3hr9737fh398f380} we find,
\begin{align}
\frac{\pt f_{VU}}{\pt y}f^i_{UV} = - f_{VU}^i\zeta_{UV};
\label{emfiojeiof309fj0j3}
\end{align}
and
\begin{align}
\frac{\pt f_{VU}}{\pt y}g^{12}_{UV} = - g_{VU}^{12} 
- f^1_{VU}\psi^2_{UV} + f^2_{VU}\psi^1_{UV}.
\label{3j903jf093jf093j}
\end{align}
In using \eqref{ojfoeofhoifjpie} and \eqref{emfiojeiof309fj0j3} we will straightforwardly deduce \eqref{keiohuieuksciwni} from \eqref{3j903jf093jf093j}. Regarding the odd component: as we are assuming $\rho_\xi$ is superconformal, we know that \eqref{648379fh93huf} must hold. Now as we have justified \eqref{keiohuieuksciwni}, we are at liberty to use this in what follows. Firstly, from \eqref{emfiojeiof309fj0j3} and \eqref{ojfoeofhoifjpie} we have
\begin{align}
\psi^i_{UV} =-\zeta_{UV}\psi_{VU}^i.
\end{align}
Now starting from \eqref{648379fh93huf}:
\begin{align*}
2\zeta_{UV}\zeta_{UV}^{12}
=&~ 
\frac{\pt g_{UV}^{12}}{\pt x} 
+
\frac{\pt \psi_{UV}^1}{\pt x}\psi_{UV}^2 
-
\psi_{UV}^1\frac{\pt \psi^2_{UV}}{\pt x}
\\
=&~
-2\zeta_{UV}\frac{\pt \zeta_{UV}}{\pt x}g_{VU}^{12} 
- 
\zeta_{UV}^4\frac{\pt g_{VU}^{12}}{\pt y}
\\
&~+
\zeta_{UV}\frac{\pt \zeta_{UV}}{\pt x}\psi_{VU}^1\psi^2_{VU} 
+ 
\zeta_{UV}^4\frac{\pt\psi^1_{VU}}{\pt y}\psi^2_{VU}
\\
&~-
\zeta_{UV}\frac{\pt \zeta_{UV}}{\pt x}\psi^1_{VU}\psi_{VU}^2 
-
\zeta_{UV}^4\psi_{VU}^1\frac{\pt \psi^2_{VU}}{\pt y}
\\
=&~
-
2\zeta_{UV}\frac{\pt \zeta_{UV}}{\pt x}g_{VU}^{12} 
-
\zeta_{UV}^4
\left(
\frac{\pt g_{VU}^{12}}{\pt y}
-
\frac{\pt\psi^1_{VU}}{\pt y}\psi^2_{VU}
+
\psi_{VU}^1\frac{\pt \psi^2_{VU}}{\pt y}
\right)
\\
=&~
-2\zeta_{UV}^3\zeta_{VU}^{12}
-
2\zeta_{UV}\frac{\pt \zeta_{UV}}{\pt x}g_{VU}^{12}.
\end{align*}
The identity in \eqref{ejioheovhhvo3f332d} now follows. 
\end{proof}

\begin{COR}\label{ueievbhjgbvbeiuno}
If the data $(\Ufr, \vartheta)$ from Lemma $\ref{encioeuifhe7f78}$ defines an atlas for an odd, second order deformation $\Xc$ of a super Riemann surface $\Scl$, then
\begin{align}
\frac{\pt g^{12}_{UV}}{\pt x}
=
2\zeta_{UV} \zeta_{UV}^{12}.
\label{3gf73g6f8g38h3}
\end{align}
\end{COR}

\begin{proof}
We want to show:
\begin{align}
\frac{\pt\psi^1_{VU}}{\pt y}\psi^2_{VU} 
&=
\frac{\pt \psi^2_{VU}}{\pt y}\psi^1_{VU}.
\label{gvhdcsvtvyiwbou39}
\end{align}
Recall that the relation for $\zeta^{12}$ on intersections in \eqref{ejioheovhhvo3f332d} was obtained by appealing to the equation \eqref{648379fh93huf} characterising superconformality. In appealing directly to the transition data and enforcing \eqref{eocneioeoeioe} however, we will obtain another relation for $\zeta^{12}$ on intersections. It is:
\begin{align}
\zeta_{UV}^{12}
&=
-
\zeta_{UV}^2~\zeta^{12}_{VU}
-
\frac{\pt\zeta_{UV}}{\pt x}g_{VU}^{12}
-
\zeta_{UV}^4
\left(
\frac{\pt\psi^1_{VU}}{\pt y}\psi^2_{VU}
-
\psi^1_{VU}
\frac{\pt \psi^2_{VU}}{\pt y}
\right)
\label{erhehh38h8dh383huefrnuyy}
\end{align}
In comparing \eqref{erhehh38h8dh383huefrnuyy} with \eqref{ejioheovhhvo3f332d} the identity in \eqref{gvhdcsvtvyiwbou39} follows. 
\end{proof}

We turn now to \eqref{eocneioeoeioe} concerning triple intersections.

\subsubsection{On Triple Intersections}

\begin{LEM}\label{fnbfhbvhbvhbhvrhjvb}
Let $\vartheta = \{\rho_\xi\}$ where $\rho_\xi = \{\rho_{\Uc\Vc;\xi}\}$ is superconformal. If $(\Ufr,\vartheta)$ defines an atlas for a supermanifold $\Xc$, then $f^i = \{f^i_{UV}\}$ and $\psi^i = \{\psi^i_{UV}\}$ define $1$-cocycles valued in $\Tc_C^{1/2}$. 
\end{LEM}

\begin{proof}
This follows immediately from expanding \eqref{eocneioeoeioe}. For $f^i$ we find,
\begin{align}
f^i_{UW}
=
\frac{\pt f_{VW}}{\pt y}f_{UV}^i + \zeta_{UV}f^i_{VW}
&&
\mbox{or}
&&
f^i_{UW}\q\frac{\pt}{\pt z}
=
f^i_{UV}\q\frac{\pt}{\pt y}
+
f^i_{VW}\eta\frac{\pt}{\pt z}
\label{efoih48fh9hf93jf03}
\end{align}
where $z$ denotes the local even coordinate on $W$. The latter expression shows that $f^i$ will be a $1$-cocycle. Similarly, for $\psi^i$ we will find
\begin{align}
\psi^i_{UW}\frac{\pt}{\pt \gam}
=
\psi^i_{UV}\frac{\pt}{\pt \eta}
+
\psi^i_{VW}\frac{\pt}{\pt \gam}.
\label{emoij48fjjf03jf0j30}
\end{align}
for $\gam$ the local odd coordinate on $W$. Hence $\psi^i$ will be a 1-cocycle also. Alternatively, given \eqref{efoih48fh9hf93jf03} and the superconformal conditions in \eqref{ojfoeofhoifjpie}, we will recover \eqref{emoij48fjjf03jf0j30} and vice-versa which shows that they must both be 1-cocycles valued in the same sheaf, which is $\Tc_C^{1/2}$.
\end{proof}

Perhaps more interesting is the 1-cochain $g^{12} = \{g^{12}_{UV}\}$. From \eqref{keiohuieuksciwni} in Lemma \ref{encioeuifhe7f78} and superconformality we know that $g^{12}\in C^1(\Ufr, \Tc_C)$. In order for it to be a $\Tc_C$-valued 1-cocycle however, it is necessary for $(\dt g^{12})_{UVW} = 0$. We will see that this will not be true in general.

\begin{PROP}\label{kvbbvhbvhjbyur}
Let $\vartheta = \{\rho_\xi\}$ where $\rho_\xi = \{\rho_{\Uc\Vc;\xi}\}$ is superconformal. If $(\Ufr,\vartheta)$ defines an atlas for a supermanifold $\Xc$, then on $U\cap V\cap W$
\begin{align}
g^{12}_{UW}
\frac{\pt}{\pt z}
-
g^{12}_{UV}\frac{\pt}{\pt y}
-
g^{12}_{VW}\frac{\pt}{\pt z}
=
(\psi^1\otimes\psi^2 - \psi^2\otimes\psi^1)_{UVW}
\label{hbvjbjvyvy4biu393}
\end{align}
\end{PROP}

\begin{proof}
Expanding \eqref{eocneioeoeioe} we find
\begin{align*}
g^{12}_{UW}
&=
\frac{\pt f_{VW}}{\pt y}g^{12}_{UV} + g^{12}_{VW}
+
\psi^2_{UV}f^1_{VW}
-
\psi^1_{UV} f^2_{VW}
\\
&=
\frac{\pt f_{VW}}{\pt y}g^{12}_{UV} + g^{12}_{VW}
+
\zeta_{VW}
\left(
\psi^2_{UV}\psi^1_{VW}
-
\psi^1_{UV}\psi^2_{VW}
\right).
\end{align*}
Now from the general theory of tensor-products of sheaves of abelian groups $\Ac$ and $\Bc$ we have on the cochains a map $C^p(\Ufr, \Ac)\otimes C^q(\Ufr, \Bc) \ra C^{p+q}(\Ufr, \Ac\otimes\Bc)$. This map is given by concatenation and so on the 1-cocycles $\psi^1, \psi^2$ we have
\begin{align}
(\psi^1\otimes\psi^2)_{UVW}
&=
\psi^1_{UV}\psi^2_{VW}\frac{\pt}{\pt \eta}\otimes\frac{\pt}{\pt \gam}
\notag
\\
&=
\zeta_{VW}\psi^1_{UV}\psi^2_{VW}\frac{\pt}{\pt \gam}\otimes\frac{\pt}{\pt \gam}
\notag
\\
&\equiv
\zeta_{VW}\psi^1_{UV}\psi^2_{VW}
\frac{\pt}{\pt z}
\label{ejkneufbu3hf3hf83f80}
\end{align}
where in \eqref{ejkneufbu3hf3hf83f80} we have used the isomorphism $\Tc_C^{1/2}\otimes\Tc^{1/2}_C\cong \Tc_C$ to identify $(\pt/\pt \gam)\otimes(\pt/\pt \gam)$ with $\pt/\pt z$ over $W$. The present proposition now follows.
\end{proof}

\begin{REM}\emph{The constraint in \eqref{hbvjbjvyvy4biu393} can always be satisfied on Riemann surfaces since: (1) we know from the proof of Proposition \ref{kvbbvhbvhjbyur} that the right-hand side of \eqref{hbvjbjvyvy4biu393} will be a $2$-cocycle; (2) that the left-hand side of \eqref{hbvjbjvyvy4biu393} is the coboundary of a $1$-cochain; and (3)
that $H^2(C,\Ac) = 0$ for any sheaf of abelian groups $\Ac$ for dimensional reasons.\footnote{\label{fh89hf983jf98j3}If we were in a more general setting where the object under deformation had (even) dimension greater than one, the bracket $[\psi_1, \psi_2]$ on the right-hand side of \eqref{hbvjbjvyvy4biu393} would represent an obstruction to the existence of a deformation.} }\end{REM}

Regarding $\zeta^{12}$, there is little to say about this object on triple intersections. For our purposes in this article, the characterisation in Corollary \ref{ueievbhjgbvbeiuno} will be sufficient. We turn now to the construction of the Kodaira-Spencer map for these deformations.

\subsection{The Kodaira-Spencer Map}
An odd, $n$-th order deformation $\Xc$ has dimension $(1|n+1)$. It will admit a superconformal atlas $(\Ufr, \vartheta)$, for $\vartheta = \{\rho_\xi\}_{\xi\in \Cbb^{0|n}}$ and:
\begin{align}
 \rho_{\Uc\Vc; \xi}^+ &= f_{UV} + f^i_{UV}~\xi_i\q + f^{ijk}_{UV}~\xi_{ijk}\q + \ldots + g^{ij}_{UV}~\xi_{ij} + \ldots
 ~~
 \mbox{and}
 \label{hbbubeubcdioiorjf4jf984jf904j4fj389f839jf9}
 \\
 \rho_{\Uc\Vc,\xi}^- &= \zeta_{UV}~\q + \zeta_{UV}^{ij}~\xi_{ij}\q + \ldots +
  \psi^i_{UV}~\xi_i + \psi_{UV}^{ijk}~\xi_{ijk} + \ldots 
 \label{dioiorjf4jf984jf904j4fj389f839jf9}
 \end{align}
From this atlas we can construct the following vector on each intersection:
\begin{align}
\mathrm{KS}_{\Uc\Vc;0} &= \left(\mathrm{KS}_{\Uc\Vc;0}^1, \ldots, \mathrm{KS}^n_{\Uc\Vc;0}\right)
\label{eiojeofh80jvepoeo}
\end{align}
where 
\begin{align}
\mathrm{KS}_{\Uc\Vc;0}^a
\notag
&= \left.(\vartheta_{\Uc\Vc})_*\left(\frac{\pt}{\pt \xi_a}\right)\right|_{\xi = 0}\\
&= f^a_{UV}(x)~\q\frac{\pt}{\pt y} + \psi^a_{UV}\frac{\pt}{\pt \eta}.
\label{ookpj02j0j0382u80jd8PMOmrm}
\end{align} 
In light of Lemma \ref{fnbfhbvhbvhbhvrhjvb}, which clearly holds for deformations of any order, we see that the $1$-cochain $\{\mathrm{KS}_{\Uc\Vc;0}^a\}$ will be a $1$-cocycle on $\Ufr$ valued in $\Tc_C^{1/2}$. The Kodaira-Spencer class of $\Xc$ is defined as follows.

\begin{DEF}
\emph{Let $\Xc$ be a deformation of a super Riemann surface $\Scl$. The $a$-th Kodaira-Spencer class of $\Xc$, denoted $\mathrm{KS}^a_0(\pi_\Xc)(\pt/\pt\xi)$, is the class 
\[
\mathrm{KS}^a_0(\pi_\Xc)\left(\frac{\pt}{\pt \xi}\right) := \left[\varinjlim_\Ufr\left\{\left.\frac{\pt \vartheta_{\Uc\Vc}}{\pt \xi_a}\right|_{\xi = 0}\right\}\right]
\]
for $(\Ufr, \vartheta)$ is a superconformal atlas of $\Xc$.\footnote{The limit over $\Ufr$ is taken over common refinement on the underlying space $\Xc_\red$. Indeed, for the kinds of supermanifolds considered in this article, their topology is concentrated in their reduced part.} The square brackets $[-]$ denote `cohomology class'. We denote by $\mathrm{KS}_0(\pi_\Xc)(\frac{\pt}{\pt \xi})$ the vector $(\mathrm{KS}^a_0(\pi_\Xc)(\frac{\pt}{\pt \xi_a}))$.}
\end{DEF}

It is more natural to study the Kodaira-Spencer \emph{map}, $\mathrm{KS}_0(\pi_\Xc)$, which is the following $\Cbb$-linear map,
\begin{align*}
\mathrm{KS}_0(\pi_\Xc) : T_0\Cbb^{0|n} \lra H^1(C, \Tc_C^{1/2})^{\oplus n}
&&\mbox{given by}&&
\frac{\pt}{\pt \xi} \longmapsto \mathrm{KS}_0(\pi_\Xc)\left(\frac{\pt}{\pt \xi}\right). 
\end{align*}
For odd, infinitesimal deformations $\Xc$, the image of the Kodaira-Spencer map can be identified with the primary obstruction class of $\Xc$. This can be inferred from Donagi and Witten's example of a non-split deformation in \cite[pp. 32-3]{DW1} and is discussed in more detail in \cite[Ch. 6]{BETTPHD}. In particular, we find: \emph{if $\Xc$ is an odd, infinitesimal deformation, it will be split if and only if its Kodaira-Spencer map vanishes}. A natural question concerns the analogous statement for deformations of higher order. 

We shall see in the following section that the analogous statement will not hold in general.

\subsection{Non-Split Deformations of Second Order}
As a result of our deliberations so far, we can be assured that the following construction of an atlas will define an odd deformation of second order since it will be superconformal.

\begin{CON}\label{enconui4fj4oif3pk3}
Let $\Theta\in H^1(C, \Tc_C^{1/2})$ be fixed and denote by $\{\Theta_{UV}\}$ a $1$-cocycle representative with respect to a cover $\Ufr$ of $C$. Consider data $(\Ufr, \vartheta)$, where $\vartheta  = \{\rho_\xi\}$ is given by,
\begin{align*}
\rho_{\Uc\Vc; \xi}^+ = f_{UV} + \frac{1}{2}\zeta_{UV}\Theta_{UV}(\xi_1+\xi_2)\q;
&&
\mbox{and}
&&
\rho_{\Uc\Vc,\xi}^- = \zeta_{UV}~\q +
\frac{1}{2}\Theta_{UV}(\xi_1 + \xi_2).
\end{align*}
Suppose $f^\p = \zeta^2$. Then the data $(\Ufr,\vartheta)$ will define an atlas for an odd, second order deformation $\Xc$ of a super Riemann surface $\Scl$. Moreover, by \eqref{ookpj02j0j0382u80jd8PMOmrm} we see that $\mathrm{KS}^i_0(\pi_\Xc) = \Theta$ for $i = 1, 2$. 
\end{CON}

\begin{LEM}\label{f64gf684f3h9}
Let $C$ be a Riemann surface with $H^1(C, \Tc_C)\neq0$. Then there exist non-split, second order deformations of $\Scl = \Pi \Om^{1/2}_C$ whose Kodaira-Spencer map vanishes.
\end{LEM}

\begin{proof}
Consider the example of an odd, second order deformation $\Xc$ in Construction \ref{enconui4fj4oif3pk3}. Observe for the superconformal atlas $(\Ufr, \vartheta)$ for $\Xc$ that the right-hand side of \eqref{hbvjbjvyvy4biu393} will vanish identically. Hence if $g^{12} = \{g^{12}_{UV}\}$ is a $1$-cocycle which defines a non-trivial element in $H^1(C, \Tc_C)$,  the data $\tilde\vartheta = \{\tilde\rho_\xi\}$ given by
\begin{align*}
\tilde\rho^{+}_{\Uc\Vc}
=
\rho^+_{\Uc\Vc} + g^{12}_{UV}\xi_{12}
&&
\mbox{and}
&&
\tilde\rho^-_{\Uc\Vc}
=
\rho_{\Uc\Vc}^- + \frac{1}{2}\zeta_{UV}^{-1}\frac{\pt g_{UV}^{12}}{\pt x}\xi_{12}\q
\end{align*}
will define a superconformal atlas for another odd, second order deformation $\widetilde\Xc$. From the expression of the Kodaira-Spencer class in \eqref{ookpj02j0j0382u80jd8PMOmrm} note that if $\Xc$ is such that $\mathrm{KS}_0(\pi_\Xc) = 0$, then $\mathrm{KS}_0(\pi_{\widetilde\Xc}) = 0$. 
 Take representatives $\{\Theta_{UV}\}= 0$. Then from Lemma  \ref{pekpfk49fk94p} see that $g^{12}$ will be a cocycle representative for the primary obstruction class of $\widetilde\Xc$. Since we assume $g^{12}\nsim 0$, we see that $\widetilde\Xc$ will be an odd, second order deformation of $\Scl$ whose Kodaira-Spencer map vanishes and yet is non-split.
\end{proof}

\section{Splitting Odd Second Order Deformations}
\label{r83jf893jf0j30}

In this section we look to relate the Kodaira-Spencer map of a second order deformation $\Xc$ with its primary obstruction class.

\subsection{Exterior Powers and the Obstruction Class}
Our starting point is the following preliminary result regarding exterior powers of sheaves of modules from \cite[pp. 127-8] {HARTALG}:

\begin{LEM}\label{f89hf89h03jfp}
Let $(X, \Oc_X)$ be a locally ringed space and suppose $\Ac\hookrightarrow \Ec\twoheadrightarrow \Bc$ is a short exact sequence of locally free sheaves of $\Oc_X$-modules. Then the $k$-th exterior power $\wedge^k\Ec$ admits a finite filtration,
\[
\wedge^k\Ec \supseteq \Fc^1\supseteq \Fc^2\supseteq \cdots\supseteq\Fc^k\supseteq 0
\]
with successive quotients satisfying
\[
\Fc^l/\Fc^{l+1} \cong \wedge^l\Ac\otimes\wedge^{k-l}\Bc
\]
for $l =0, \ldots, k$. \qed
\end{LEM}

Now fix an odd, second order deformation $\Xc$ of a super Riemann surface $\Scl = \Pi \Om^{1/2}_C$ and denote by $\Ec$ the modelling bundle of $\Xc$. As $\Xc$ is $(1|3)$-dimensional, we know that $\Ec$ will have rank $3$. From Lemma \ref{fnbfhbvhbvhbhvrhjvb} we can infer that $\Ec$ will fit into the following short exact sequence,
\[
0 \ra \Cc_C^{\oplus 2} \hookrightarrow \Ec \twoheadrightarrow \Om^{1/2}_C\ra0.
\]
where $\Cc_C$ is the structure sheaf of the underlying Riemann surface $C$. Regarding the primary obstruction class of $\Xc$, recall from Lemma \ref{pekpfk49fk94p} that it will be valued in $H^1(C, \wedge^2\Ec\otimes\Tc_C)$. Hence we are interested here in $\wedge^2\Ec$. From Lemma $\ref{f89hf89h03jfp}$ we see that $\wedge^2\Ec$ will admit a finite filtration $\wedge^2\Ec \supseteq \Fc^1\supset \Fc^2\supseteq0$ with successive quotients,
\begin{align*}
\Fc^0/\Fc^1
= 0
&&
\Fc^1/\Fc^2\cong \Om^{1/2}_C \oplus \Om^{1/2}_C;
&&
\mbox{and}
&&
\Fc^2\cong \Cc_C. 
\end{align*}
This leads to the following short exact sequence,
\begin{align*}
0 \ra \Cc_C\hookrightarrow \wedge^2\Ec \twoheadrightarrow  \Om^{1/2}_C \oplus \Om^{1/2}_C\ra0
\end{align*}
which in turn gives
\begin{align}
0 \ra \Tc_C\stackrel{\iota}{\hookrightarrow} \wedge^2\Ec\otimes\Tc_C \stackrel{p}{\twoheadrightarrow}  \Tc^{1/2}_C \oplus \Tc^{1/2}_C\ra0. 
\label{nNUIBFIUNoi}
\end{align}
Denote by $\iota_*$ and $p_*$ the corresponding induced maps on cohomology. 

Before presenting the next result, it will be useful to define the following: on each non-empty intersection $\Uc\cap \Vc$ set,
\begin{align*}
g_{\Uc\Vc} := \frac{\pt^2\vartheta^+_{\Uc\Vc}}{\pt \xi_2\pt\xi_1}\otimes \frac{\pt}{\pt y}.
\end{align*}
As defined, we have $g_{(\Ufr, \vartheta)}:=\{g_{\Uc\Vc}\} = \{g_{UV}^{12}\pt/\pt y\}$. It is a $\Tc_C$-valued 1-cochain on $\Ufr$. We have included the subscript $(\Ufr, \vartheta)$ to emphasize that this object generally depends on the atlas (see Proposition \ref{kvbbvhbvhjbyur}).

\begin{PROP}\label{5g4hg74jf989f3}
Let $\Xc$ be an odd, second order deformation of a super Riemann surface $\Scl$ with genus $g>1$ and primary obstruction $\om_\Xc$. Then,
\begin{enumerate}[(i)]
	\item 
	$p_*\om_\Xc = \img\left(\frac{1}{2}\mathrm{KS}_0(\pi_\Xc)\right)$; and
	\item If $p_*\om_\Xc= 0$, then in any superconformal atlas $(\Ufr, \vartheta)$ for $\Xc$, we have $\dt g_{(\Ufr, \vartheta)} = 0$. 
	\end{enumerate}
\end{PROP}

\begin{proof}At the level of cocycles with respect to a covering, recall the expression for the cocycle representative of the obstruction class $\om_\Xc$ in \eqref{TVYTRVUYEBIY} in Lemma \ref{pekpfk49fk94p}. Comparing this with the cocycle representative for the image of the Kodaira-Spencer class in \eqref{ookpj02j0j0382u80jd8PMOmrm} and using the relation in \eqref{ojfoeofhoifjpie} from superconformality confirms \emph{(i)}. As for \emph{(ii)}, we know by exactness of the long exact sequence on sheaf cohomology that $\om_\Xc$ will be in the image of \emph{some} class $\al$ in $H^1(C, \Tc_C)$. Given $\al$, one can take any cocycle representative of $\al$ with respect to a covering $\Ufr$, say $\{\al_{UV}\}$, and write down an atlas $(\Ufr, \vartheta)$ for $\Xc$ as in, for instance, the proof of Lemma \ref{f64gf684f3h9}. The subtlety lies in the converse statement: given any atlas $(\Ufr, \vartheta)$ for $\Xc$ with $\vartheta$ superconformal, there will generally be an obstruction to solving $\dt g_{(\Ufr, \tilde\vartheta)} = 0$ as can be seen from Proposition \ref{kvbbvhbvhjbyur}. We will argue however that this obstruction must naturally vanish. 

Firstly, by \emph{(i)} we see $p_*\om_\Xc = 0$ means $\mathrm{KS}_0(\pi_\Xc) = 0$. Hence for representatives $\psi^i = \{\psi^i_{UV}\}$ of the Kodaira-Spencer class in a covering $\Ufr$ we have $\psi^i = \dt\s^i$, for some $0$-cochain $\s^i\in C^0(\Ufr, \Tc_C^{1/2})$. We now have the following computation on triple intersections $U\cap V\cap W$:
\begin{align}
[\psi^1, \psi^2]_{UVW} &= 
[\dt \s^1, \dt\s^2]_{UVW} 
\notag
\\
&=
(\s^1_V - \s^1_U)\otimes (\s^2_W - \s^2_V) - (\s^2_V - \s^2_U)\otimes (\s^1_W - \s^1_V)
\notag
\\
&=
[\s^1, \s^2]_{UW} - [\s^1, \s^2]_{UV} - [\s^1, \s^2]_{VW}
\label{eiurugug783}
\end{align}
where $[\s^1, \s^2]_{UV} := \s^1_U\otimes \s^2_V - \s^2_U\otimes \s^1_V$. Our objective will be to argue that in fact $[\s^1, \s^2]_{UV} = 0$. To see this it will be convenient to write $\psi^i= \dt\s^i$ more explicitly. We have,
\begin{align*}
\psi^i_{UV} \frac{\pt}{\pt \eta}
=
\s^i_V\frac{\pt}{\pt \eta}
-
\s^i_U\frac{\pt}{\pt \q}
=
\left(\s_V^i - \zeta_{UV}\s^i_U\right)\frac{\pt}{\pt \eta}.
\end{align*}
Hence $\psi^i_{UV} = \s_V^i - \zeta_{UV}\s^i_U$. Now since $(\Ufr, \vartheta)$ is superconformal, we have the identity from Corollary \ref{ueievbhjgbvbeiuno} which we recall below for convenience:
\begin{align}
\frac{\pt \psi^1_{UV}}{\pt x}\psi^2_{UV}
=
\psi^1_{UV}\frac{\pt \psi^2_{UV}}{\pt x}.
\label{4f4fh4jf83j03j0}
\end{align}
This gives
\begin{align*}
\frac{\pt \psi^i_{UV}}{\pt x} 
=
\zeta^2_{UV}\frac{\pt \s_V^i}{\pt y}
- \frac{\pt \zeta_{UV}}{\pt x}\s^i_U 
- \zeta_{UV}\frac{\pt \s^i_U}{\pt x}.
\end{align*}
From \eqref{4f4fh4jf83j03j0} we then find:
\begin{align}
\frac{\pt\zeta_{UV}}{\pt x}
\left( \s^2_U\s^1_V - \s^1_U\s^2_V
\right)
+ O(\zeta_{UV})
=
0
\label{48f89jfj30j30}
\end{align}
where $O(\zeta_{UV})$ denotes terms proportional to powers of $\zeta_{UV}$. 

Now recall that $\zeta$ are the transition functions for the spin structure $\Om_C^{1/2}$. When $C$ has genus $g>1$, this bundle will be non-trivial and so $\zeta$ will be non-constant. Therefore the powers of $\zeta$ will be algebraically independent and so, in order to impose \eqref{48f89jfj30j30}, it is necessary to have $[\s^1, \s^2]_{UV} = 0$. As \eqref{4f4fh4jf83j03j0} must apply in every non-empty intersection $U\cap V$, we see that $[\s^1,\s^2] = 0$ and therefore $[\psi^1, \psi^2] = 0$ by \eqref{eiurugug783}. Part \emph{(ii)} of this proposition now follows from Proposition \ref{kvbbvhbvhjbyur}. 
\end{proof}

We learn from Proposition \ref{5g4hg74jf989f3} above that in the case where the Kodaira-Spencer map of the deformation $\Xc$ vanishes, the object $g_{(\Ufr, \vartheta)}$ will not depend on the atlas $(\Ufr, \vartheta)$. Hence $\varinjlim_\Ufr g_{(\Ufr, \vartheta)}$ will define a cohomology class $g(\pi_\Xc)\in H^1(C, \Tc^{1/2}_C)$. From the exact sequence in \eqref{nNUIBFIUNoi} we see that $\om_\Xc = \iota_* g(\pi_\Xc)$. 
We conclude now with the following result concerning the vanishing of $\om_\Xc$.

\begin{PROP}\label{fnuuieuf3h9u}
(i) If $\mathrm{KS}_0(\pi_\Xc) = 0$ and $g(\pi_\Xc) = 0$, then $\om_\Xc = 0$; (ii) Conversely, if $\Xc$ is an odd, second order deformation of a super Riemann surface $\Scl$ of genus $g> 1$, then $\om_\Xc = 0$ implies both $\mathrm{KS}_0(\pi_\Xc) = 0$ and $g(\pi_\Xc) = 0$.
\end{PROP}

\begin{proof}
Part $(i)$ of this corollary follows immediately from Proposition \ref{5g4hg74jf989f3}. As for part $(ii)$, that $\om_\Xc = 0$ implies $\mathrm{KS}_0(\pi_\Xc) = 0$ is clear. As for $\om_\Xc = 0$ implying the vanishing of $g(\pi_\Xc)$, this follows from the fact that $H^0(C, \Tc_C^{1/2}) = 0$ in genus $g> 1$ since here $\Tc_C^{1/2}$ will be a line bundle of negative degree. 
This fact implies in particular that the map induced on cohomology $\iota_* : H^1(C, \Tc_C)\ra H^1(\wedge^2\Ec\otimes\Tc_C)$ from the short exact sequence in \eqref{nNUIBFIUNoi} will be injective. Hence if $\om_\Xc = \iota_*g(\pi_\Xc)$ and $\om_\Xc$ vanishes, $g(\pi_\Xc)$ must vanish also. 
\end{proof}

\subsection{On Splitting Deformations of Second Order}

We conclude this article with the following result and its proof, concerning sufficiency conditions on splitting a deformation of second order. 

\begin{THM}\label{tgguygvg3fg83h3}
Let $\Xc \stackrel{\pi_\Xc}{\ra}\Cbb^{0|2}$ be a deformation of a super Riemann surface $\Scl$ with genus $g> 1$. Then $\Xc$ is split if and only if its primary obstruction class, $\om_\Xc$, vanishes.
\end{THM}

\noindent
\emph{Proof}.
If $\Xc$ is split then obviously the obstruction class to splitting $\Xc$ vanishes. The main difficulty in proving the converse is in ensuring, upon assuming $\om_\Xc $ vanishes, that the \emph{second} obstruction class to splitting necessarily vanishes also. We have not discussed higher obstruction theory in this article as it would be superfluous. Instead, in what follows, we will just construct an explicit splitting map. 

Our method is to show, under the assumption $\om_\Xc =0$ and the genus $g>1$, that any superconformal atlas can be split.

\subsubsection{Preliminaries: The Splitting Equations}
We begin with the following preliminary result which allows us to justify our methods:

\begin{LEM}\label{oopejfioehuhveoij}
Let $\Xfr$ be a supermanifold.  Then $\Xfr$ is split if and only if there exists:
\begin{enumerate}[(i)]
	\item an atlas $(\Ufr, \rho)$ for $\rho$ as in \eqref{uicuibviuv78v3} and \eqref{porviorhg894h89} and;
	\item a $0$-cochain $\Lam = \{\Lam_\Uc\}_{\Uc\in \Ufr}$ of automorphisms;
\end{enumerate}
such that
 on all non-empty intersections $\Uc\cap\Vc$:
\begin{align}
\Lam_\Vc\circ \rho_{\Uc\Vc} = \hat\rho_{\Uc\Vc}\circ \Lam_\Uc
\label{dj90jf390jf903k93}
\end{align}
where $\{\hat\rho_{\Uc\Vc}\}$ are the transition functions of the split model of $\Xfr$:
\begin{align*}
\hat \rho^\mu_{\Uc\Vc}(x,\q) = f_{UV}^\mu(x)
&&\mbox{and}&&
\hat \rho_{\Uc\Vc;a}(x,\q)
=
\zeta_{UV, a}^b(x) ~\q_b.
\end{align*}
The $0$-cochain $\Lam$ is referred to as a splitting map.
\end{LEM}

\begin{proof}
This lemma is in essence a restatement of the original classification of supermanifolds by Green in \cite{GREEN}, using \v Cech cohomology of non-abelian groups. We omit the details here.
\end{proof}

We will firstly describe a slight generalisation of Lemma \ref{oopejfioehuhveoij} above. Given an atlas $(\Ufr, \vartheta)$, we can write the even and odd components as
\begin{align}
\rho_{\Uc\Vc;\xi}^+
&=
f_{UV} + f^i_{UV}\xi_i\q +  g_{UV}^{12}\xi_{12};~\mbox{and}
\label{efjejf894jf89j4}
\\
\rho^-_{\Uc\Vc;\xi}
&=
\zeta_{UV}\q + \psi^i_{UV}\xi_i +  \zeta_{UV}^{12}\xi_{12}.
\label{efiojfjf98jf93}
\end{align}
We aim to compare this with another atlas $(\Ufr, \tilde\vartheta)$, where $\tilde\vartheta = \{\tilde\rho_{\Uc\Vc;\xi}\}$ is given by 
\begin{align*}
\tilde\rho_{\Uc\Vc;\xi}^+
&=
f_{UV} +\tilde f^i_{UV}\xi_i\q + \tilde g_{UV}^{12}\xi_{12};~\mbox{and}
\\
\tilde\rho^-_{\Uc\Vc;\xi}
&=
\zeta_{UV}\q + \tilde\psi^i_{UV}\xi_i +  \tilde\zeta_{UV}^{12}\xi_{12}.
\end{align*}
Let $\Lam = \{\Lam_\Uc\}$ be a $0$-cochain of automorphisms of $\Ufr$, where
\begin{align*}
\Lam_\Uc^+
= x + \lam_U^i\xi_i\q + \lam_U^{12}\xi_{12}
&&
\mbox{and}
&&
\Lam^-_\Uc
=
\q + \phi_U^i\xi_i + \phi_U^{12}\xi_{12}\q.
\end{align*}
The atlases $(\Ufr, \vartheta)$ and $(\Ufr, \tilde\vartheta)$ will be equivalent, i.e., describe isomorphic supermanifolds, iff there exists a $0$-cochain $\Lam$ as above such that 
\begin{align}
\Lam_\Vc^\pm\circ \rho_{\Uc\Vc} = \tilde\rho_{\Uc\Vc}^\pm\circ \Lam_\Uc
\label{r8f8489fj38fj30}
\end{align}
on all non-empty intersections $\Uc\cap\Vc$.

In expanding the left- and right-hand sides of \eqref{r8f8489fj38fj30} respectively we find, for the even component:
\begin{align*}
\mathrm{LHS}_{\eqref{r8f8489fj38fj30}^+}
&=
f_{UV} + \left( f_{UV}^j + \zeta_{UV}\lam^j_V\right)\xi_j\q
+
\left( 
\left(\lam^1_V\psi^2_{UV} - \lam^2_V\psi^1_{UV}\right)
+ g^{12}_{UV}+ \lam^{12}_V
\right)\xi_{12};
\\
\mathrm{RHS}_{\eqref{r8f8489fj38fj30}^+}
&=
f_{UV}
+
\left(\frac{\pt f_{UV}}{\pt x}\lam^j_U + \tilde f_{UV}^j\right)\xi_j\q
+
\left(
\tilde f^1_{UV}\phi^2_U - \tilde f_{UV}^2\phi^1_U
+
\tilde g^{12}_{UV}
+
\frac{\pt f_{UV}}{\pt x}\lam^{12}_U 
\right)
\xi_{12}.
\end{align*}
And for the odd component we have:
\begin{align*}
\mathrm{LHS}_{\eqref{r8f8489fj38fj30}^-}
&=
\zeta_{UV}\q + \left(\psi^j_{UV} + \phi_V^j\right)\xi_j
\\
&~-
\left(
\left(
\frac{\pt \phi^1_V}{\pt x}f^2_{UV} - \frac{\pt \phi^2_V}{\pt x}f^1_{UV}
\right)
-\zeta^{12}_{UV}
-\zeta_{UV}\phi^{12}_V
\right)\xi_{12}\q;
\\
\mathrm{RHS}_{\eqref{r8f8489fj38fj30}^-}
&=
\zeta_{UV}\q
+
\left(\zeta_{UV}\phi^j_U + \tilde\psi_{UV}^j\right)\xi_j
\\
&~-
\left(
\left(
\frac{\pt \tilde\psi_{UV}^1}{\pt x}\lam^2_U
-
\frac{\pt\tilde\psi^2_{UV}}{\pt x}\lam^1_U
\right)
-
\tilde\zeta^{12}_{UV}
-
\frac{\pt \zeta_{UV}}{\pt x}\lam_U^{12}
-
\zeta_{UV}\phi^{12}_U
\right)
\xi_{12}\q.
\end{align*}
Then in order to solve \eqref{r8f8489fj38fj30}, we will need to solve the following collection of equations for the coefficients of the transition functions $\{\tilde\rho_{\Uc\Vc;\xi}\}$:
\begin{align}
\tilde f^j_{UV}
-
f^j_{UV}
&= \zeta_{UV}\lam^j_V - \frac{\pt f_{UV}}{\pt x}\lam^j_U;
\label{rf784hf79hf9893}
\\
\tilde \psi^j_{UV}
-
\psi^j_{UV}
&=
\phi^j_V - \zeta_{UV}\phi^j_U;
\label{eijfj3fj0303}
\\
\tilde g^{12}_{UV}
-
g^{12}_{UV}
&=
\lam^{12}_V- \frac{\pt f_{UV}}{\pt x}\lam^{12}_U\
\notag
\\
&~
+
\left(\lam^1_V\psi^2_{UV} - \lam^2_V\psi^1_{UV}\right)
-
\left(
\tilde f^1_{UV}\phi^2_U
-
\tilde f^2_{UV}\phi^1_U
\right);
\label{4f78hf94hf98f03}
\\
\tilde\zeta^{12}_{UV}
-
\zeta^{12}_{UV}
&=
\zeta_{UV}\phi^{12}_V - \zeta_{UV}\phi^{12}_U
-\frac{\pt \zeta_{UV}}{\pt x}\lam_U^{12}
\notag
\\
&~
+
\left(
\left(\frac{\pt \tilde\psi_{UV}^1}{\pt x}\lam^2_U - \frac{\pt \tilde\psi^2_{UV}}{\pt x}\lam^1_U\right)
-
\left(
\frac{\pt \phi^1_V}{\pt y}f^2_{UV} - \frac{\pt\phi^2_V}{\pt y}f^1_{UV}\right)
\right),
\label{4904390fk39k300}
\end{align}
for $j= 1, 2$. We refer to \eqref{rf784hf79hf9893}---\eqref{4904390fk39k300} as \emph{splitting equations}. The equations in \eqref{rf784hf79hf9893} and \eqref{eijfj3fj0303} are cohomological relations. In the case where these atlases correspond to deformations of super Riemann surfaces $\Xc$ and $\tilde\Xc$ respectively, \eqref{rf784hf79hf9893} and \eqref{eijfj3fj0303} are equivalent to the constraint $\mathrm{KS}_0(\pi_\Xc) = \mathrm{KS}_0(\pi_{\tilde\Xc})$. The other two equations are a little more mysterious. However, for second order deformations of super Riemann surfaces, they admit a nice simplification. 

\subsubsection{On Deformations of Second Order}
Let $\Xc$ be a second order deformation of a super Riemann surface and suppose $(\Ufr, \vartheta)$ is a superconformal atlas for $\Xc$. We assume firstly that the genus $g$ of the super Riemann surface satisfies $g>1$; and secondly that $\mathrm{KS}_0(\pi_\Xc) = 0$.
  Recall the following constraint from Corollary \ref{ueievbhjgbvbeiuno}:
\begin{align}
\frac{\pt \psi^1_{UV}}{\pt x}\psi^2_{UV}
=
\psi^1_{UV}\frac{\pt \psi^2_{UV}}{\pt x}.
\label{4f4fh4jf83j03j0}
\end{align}
Under the assumption of a vanishing Kodaira-Spencer class, this identity was used to prove Proposition \ref{5g4hg74jf989f3}. We will make use of it here also. With  $\psi^i_{UV}
=\zeta_{UV}\phi^i_U - \phi^i_V$ the constraint in \eqref{4f4fh4jf83j03j0} leads to the following equation:
\begin{align}
0
&=
\frac{\pt\zeta_{UV}}{\pt x}
\left( \phi^2_U\phi^1_V - \phi^1_U\phi^2_V
\right)
\label{4hf848f99j9}
\\
&~+
\zeta_{UV}
\left(
\frac{\pt \phi^2_U}{\pt x}
\phi^1_V
-
\frac{\pt \phi^1_U}{\pt x}\phi^2_V
\right)
\label{4rh974hf9f8j303j}
\\
&~+
\zeta^2_{UV}
\left[
\left( \frac{\pt \phi^1_U}{\pt x}\phi^2_U - \phi^1_U\frac{\pt \phi^2_U}{\pt x}\right)
+ 
\left(
\frac{\pt \phi^1_V}{\pt y}\phi^2_V
-
\phi^1_V\frac{\pt \phi^2_V}{\pt y}
\right)
\right]
\\
&~+
\zeta^3_{UV}
\left(
\phi^1_U\frac{\pt \phi^2_V}{\pt y}
-
\phi^2_U\frac{\pt \phi^1_V}{\pt y}
\right).
\label{rjf8jf83j0302}
\end{align}
As argued in Proposition \ref{5g4hg74jf989f3}, in order for the above equation to hold it is necessary for each term proportional to powers of $\zeta_{UV}$ vanish. Since \eqref{4f4fh4jf83j03j0} holds on all intersections $U\cap V$, then so do the equations \eqref{4hf848f99j9}---\eqref{rjf8jf83j0302}. In looking at these equations on a self-intersection $U\cap U$ (resp. $V\cap V$) note that \eqref{4rh974hf9f8j303j} and \eqref{rjf8jf83j0302} will imply:\begin{align}
\frac{\pt \phi^1_U}{\pt x}\phi^2_U - \phi^1_U\frac{\pt \phi^2_U}{\pt x} = 0
&&
\mbox{and}
&&
\frac{\pt \phi^1_V}{\pt y}\phi^2_V
-
\phi^1_V\frac{\pt \phi^2_V}{\pt y} = 0. 
\label{rhf784hf93j98j39}
\end{align}
The above observations have the following consequences on the splitting equations. 

Concerning the terms in \eqref{4f78hf94hf98f03}, see that:
\begin{align}
\lam^1_V\psi^2_{UV} - \lam^2_V\psi^1_{UV}
&= 
\phi_V^1 \left( \zeta_{UV}\phi^2_U - \phi^2_V\right)
-
\phi^2_V\left(  \zeta_{UV}\phi^1_U - \phi^1_V\right)
\notag
\\
&=
\zeta_{UV} \left( \phi^1_V\phi^2_U - \phi^2_V\phi^1_V\right)
\notag
\\
&= 0&&\mbox{by \eqref{4hf848f99j9}.}
\label{rf489f9fj00}
\end{align}
We have used that $\lam^1 = \phi^1$ since these are solutions to the coboundary equations in \eqref{rf784hf79hf9893} and \eqref{eijfj3fj0303} for the Kodaira-Spencer representatives $f^i$ and $\psi^i$; and moreover that we identify $f^i$ and $\psi^i$ by superconformality in \eqref{ojfoeofhoifjpie}.  

Regarding the terms in \eqref{4904390fk39k300} we have,
\begin{align}
\frac{\pt \phi^1_V}{\pt y}f^2_{UV} - \frac{\pt\phi^2_V}{\pt y}f^1_{UV}
&=
\frac{\pt \phi^1_V}{\pt y}
\left(
\zeta_{UV}^2\lam_U^2 - \zeta_{UV}\lam^2_V
\right)
-
\frac{\pt \phi^2_V}{\pt y}
\left(
\zeta_{UV}^2\lam_U^1 - \zeta_{UV}\lam^1_V
\right)
\notag
\\
&=
\zeta_{UV}
\left(
\frac{\pt \phi^1_V}{\pt y}\phi^1_V - \frac{\pt \phi^1_V}{\pt y}\phi^2_V
\right)
-
\zeta^2_{UV}
\left(
\phi^1_U\frac{\pt \phi^2_V}{\pt y}
-
\phi^2_U\frac{\pt \phi^1_V}{\pt y}
\right)
\notag
\\
&=
0,
\label{riou4gh9849j30f}
\end{align}
which follows from \eqref{rjf8jf83j0302} and \eqref{rhf784hf93j98j39}. 

\subsubsection{The Splitting Map for Deformations}
We have seen what the splitting equations are. In addition to them, the \emph{splitting conditions} are the conditions under which a splitting map will exist. Such conditions were identified in Lemma \ref{oopejfioehuhveoij} and they can be read off directly from the splitting equations \eqref{rf784hf79hf9893}---\eqref{4904390fk39k300} by taking $\tilde\vartheta$ to be split (i.e., setting $\tilde f^i = \tilde \psi^i = \tilde g^{12} = \tilde\zeta^{12} = 0$). As a corollary of our calculations in \eqref{rf489f9fj00} and \eqref{riou4gh9849j30f} we have:

\begin{COR}\label{fj4fj894jf83333}
Let $\Xc$ be an odd, second order deformation of a super Riemann surface of genus $g>1$ and suppose $(\Ufr, \vartheta)$ is a superconformal atlas for $\Xc$. The splitting conditions for this atlas are:
\begin{align}
f^j_{UV}
&= 
\frac{\pt f_{UV}}{\pt x}\lam^j_U
-
\zeta_{UV}\lam^j_V ;
\label{eneuvheiheo}
\\
\psi^j_{UV}
&=
\zeta_{UV}\phi^j_U
-
\phi^j_V ;
\\
g^{12}_{UV}
&=
 \frac{\pt f_{UV}}{\pt x}\lam^{12}_U- \lam^{12}_V;
 \label{eiehfiuehioejoiejo}
\\
\zeta^{12}_{UV}
&=
\zeta_{UV}\phi^{12}_U - \zeta_{UV}\phi^{12}_V
+
\frac{\pt \zeta_{UV}}{\pt x}\lam_U^{12}
\label{vuyrgvuyrguyvheiu}
\end{align}
\qed
\end{COR}

We begin now with the assumptions in the statement of Theorem \ref{tgguygvg3fg83h3}.
Fix an odd, second order deformation $\Xc$ of a super Riemann surface $\Scl$ of genus $g> 1$ and suppose $\om_\Xc = 0$. We want to deduce then that $\Xc$ is in fact split. In order to do so, we need to show that $\Xc$ admits a split atlas. Since $\Xc$ is a deformation of a super Riemann surface, it will admit a superconformal atlas $(\Ufr, \vartheta)$ with $\vartheta = \{\rho_{\Uc\Vc;\xi}\}$ as in \eqref{efjejf894jf89j4} and \eqref{efiojfjf98jf93}. In Corollary \ref{fj4fj894jf83333} we find the splitting conditions for $\Xc$. Observe that \eqref{eneuvheiheo}---\eqref{eiehfiuehioejoiejo} are equivalent to the assumption $\om_\Xc = 0$ by Proposition \ref{fnuuieuf3h9u}. Hence we can solve for a $0$-cochain $\Lam = \{\Lam_\Uc\}$ such that \eqref{eneuvheiheo}---\eqref{eiehfiuehioejoiejo} will hold. It remains to show that we can then solve \eqref{vuyrgvuyrguyvheiu} and hence deduce that $\Lam$ will in fact be a splitting. 

Using again the assumption that $(\Ufr, \vartheta)$ is superconformal, recall that 
\begin{align*}
\zeta_{UV}^{12}
&=
\frac{1}{2}\zeta_{UV}^{-1}\frac{\pt g^{12}_{UV}}{\pt x}
&&\mbox{from Corollary \ref{ueievbhjgbvbeiuno};}
\\
&=
\frac{1}{2}
\zeta_{UV}
\left(
\frac{\pt \lam_U^{12}}{\pt x}
-
\frac{\pt \lam_V^{12}}{\pt y}
\right)
+
\frac{\pt \zeta_{UV}}{\pt x}\lam_U^{12}
&&
\mbox{from \eqref{eiehfiuehioejoiejo}}.
\end{align*}
In comparing now with \eqref{vuyrgvuyrguyvheiu} we see that $\Lam= \{\Lam_\Uc\}$ will be a splitting iff we can solve:
\begin{align*}
\zeta_{UV}
\left(
\phi^{12}_U 
-
\frac{1}{2}\frac{\pt \lam_U^{12}}{\pt x}
\right)
=
\zeta_{UV}
\left(
\phi^{12}_V
-
\frac{1}{2}
\frac{\pt \lam_V^{12}}{\pt y}
\right).
\end{align*}
This is of course possible since the $0$-cochain $\{\phi^{12}_U\}$ is a priori unconstrained by the assumption $\om_\Xc = 0$. 

We conclude that any superconformal atlas $(\Ufr, \vartheta)$ for a deformation $\Xc\stackrel{\pi_\Xc}{\ra}\Cbb^{0|2}$ of a super Riemann surface with genus $g>1$ can be split. The splitting is explicitly,
\begin{align*}
\Lam_\Uc^+
=
x + \lam_U^i\xi_i\q
+
 \lam_U^{12}\xi_{12}
&&
\mbox{and}
&&
\Lam_\Uc^-
=
\q 
+
\lam_U^i\xi_i
+ 
\frac{1}{2}\frac{\pt \lam^{12}_U}{\pt x}\xi_{12}\q.
\end{align*}
Since $\Xc$ will always admit a superconformal atlas up to isomorphism, it follows that spltting a superconformal atlas for $\Xc$ will define an isomorphism of $\Xc$ with its split model.

This concludes the proof of Theorem \ref{tgguygvg3fg83h3}. 
\qed

\section{Concluding Remarks}

\noindent
In this article we have studied the infinitesimal and second order deformations of a super Riemann surface as complex supermanifolds. Our methods were largely \v Cech-theoretic and sheaf-cohomological. However, they can just as easily be described using a Dolbeault model. Concerning the second order deformations, such a model was presented in \cite{DW2}. There, a gauge-invariant pairing between the deformation class of $\Xc$ (which here would be related to the primary obstruction $\om_\Xc$) and the primary obstruction class of supermoduli space $\Mfr_g$ is constructed. It would be interesting to see if this pairing can be extended to deformations $\Xc$ of any order $n>2$; and whether it might be more clearly related to the primary obstruction class of $\Xc$.
\\\\
Our focus in this article has been on the primary obstruction class $\om_\Xc$. When $\Xc$ is an odd deformation of a super Riemann surface $\Scl$ and of second order, it is $(1|3)$-dimensional as a complex supermanifold. This means Lemma \ref{ejf9hf98h839f3830} does not apply and so $\om_\Xc$ need not classify $\Xc$. Hence it is perhaps surprising to find, when the genus $g$ of $\Scl$ satisfies $g> 1$, that $\om_\Xc$ will nevertheless classify $\Xc$ in the sense of Theorem \ref{tgguygvg3fg83h3}. This leads therefore to a natural question regarding deformations of higher order.

\begin{QUE}\label{fh39hf9839323}
Let $\Xc\stackrel{\pi_\Xc}{\ra}\Cbb^{0|n}$, for $n> 2$, be a deformation of a super Riemann surface $\Scl$ with genus $g> 1$. Suppose moreover that the primary obstruction class to splitting $\Xc$ vanishes. Then is $\Xc$ split as a supermanifold?
\end{QUE}

Regarding the assumption on the genus, our arguments hold more generally for genus $g\neq 1$. The main problem in genus $g = 1$ is that $\Tc_C^{1/2}$ will be a line bundle of degree zero. Hence it could be trivial in which case the argument in Proposition \ref{5g4hg74jf989f3} need no longer hold. Furthermore if $\Tc^{1/2}_C$ is trivial, then $H^0(C, \Tc^{1/2}_C)\cong \Cbb$ which means the following piece of the long exact sequence on cohomology
\begin{align*}
\cdots \lra H^0(\Tc_C^{1/2}\oplus \Tc_C^{1/2}) \lra H^1(\Tc_C) \stackrel{\iota_*}{\lra} H^1(\wedge^2\Ec\otimes \Tc_C) \lra \cdots
\end{align*}
need not imply $\iota_*$ will be injective. As such we cannot conclude $\om_\Xc = 0\Rightarrow g(\pi_\Xc) = 0$ as in Proposition \ref{fnuuieuf3h9u}. It is thus natural to pose the following question:

\begin{QUE}
Does there exist an odd, second order, non-split deformation of a genus one super Riemann surface whose primary obstruction class vanishes?
\end{QUE}

Such a deformation affirming the above question would be an interesting and instructive illustration of higher obstruction theory for supermanifolds.

\small
\bibliographystyle{alpha}
\bibliography{Bibliography}

\begin{thebibliography}{Bet16b}

\bibitem[Ber87]{BER}
F.~A. Berezin.
\newblock {\em Introduction to Superanalysis}.
\newblock D. Reidel Publishing Company, 1987.

\bibitem[Bet16a]{OBSTHICK}
K.~Bettadapura.
\newblock Obstructed {Thickenings} and {Supermanifolds}.
\newblock Available at:
  {\href{https://arxiv.org/abs/1608.07810}{arXiv:1608.07810}} [math-ph], August
  2016.

\bibitem[Bet16b]{BETTPHD}
K.~Bettadapura.
\newblock {\em Obstruction Theory for Supermanifolds and Deformations of
  Superconformal Structures}.
\newblock PhD thesis, The Australian National University,
  \href{http://hdl.handle.net/1885/110239}{hdl.handle.net/1885/110239},
  December 2016.

\bibitem[CR88]{RABCRANE}
L.~Crane and J.~Rabin.
\newblock Super {Riemann} surfaces: Uniformization and {Teichmueller} theory.
\newblock {\em Comm. Math. Phys.}, (113):601--623, 1988.

\bibitem[DM99]{QFAS}
P.~Deligne and J.~W. Morgan.
\newblock {\em Quantum Fields and Strings: A course for Mathematicians},
  volume~1, chapter Notes on Supersymmetry (following Joseph Bernstein), pages
  41--97.
\newblock American Mathematical Society, Providence, 1999.

\bibitem[DW12]{DW1}
R.~Donagi and E.~Witten.
\newblock Supermoduli space is not projected.
\newblock In {\em String-Math 16-21}, 2012.
\newblock Available at:
  {\href{http://arxiv.org/abs/1304.7798}{arXiv:1304.7798}} [hep-th].

\bibitem[DW14]{DW2}
R.~Donagi and E.~Witten.
\newblock Super {Atiyah} classes and obstructions to splitting of supermoduli
  space.
\newblock available at: \href{http://arxiv.org/abs/1404.6257}{arXiv:1404.6257}
  [hep-th], 2014.

\bibitem[FR90]{FALQMOD}
G.~Falqui and C.~Reina.
\newblock A note on the global structure of supermoduli spaces.
\newblock {\em Comm. Math. Phys.}, (128):247--261, 1990.

\bibitem[Fri86]{FRIEDAN}
D.~Friedan.
\newblock Notes on string theory and two dimensional conformal field theory.
\newblock In M.~B. Green and D.~Gross, editors, {\em Unified String Theories}.
  World Scientific, 1986.

\bibitem[Gre82]{GREEN}
P.~Green.
\newblock On holomorphic graded manifolds.
\newblock {\em Proc. Amer. Math. Soc.}, 85(4):587--590, 1982.

\bibitem[Har77]{HARTALG}
R.~Hartshorne.
\newblock {\em Algebraic Geometry}.
\newblock Springer, 1977.

\bibitem[Kod86]{KS}
K.~Kodaira.
\newblock {\em Complex Manifolds and Deformation of Complex Structures}.
\newblock Springer, 1986.

\bibitem[Lei80]{LEI}
D.~Leites.
\newblock Introduction to the theory of supermanifolds.
\newblock {\em Russian Math. Surveys}, 35(1):1--64, 1980.

\bibitem[Man88]{YMAN}
Y.~Manin.
\newblock {\em Gauge Fields and Complex Geometry}.
\newblock Springer-Verlag, 1988.

\bibitem[Oni99]{ONISHCLASS}
A.~L. Onishchik.
\newblock On the classification of complex analytic supermanifolds.
\newblock {\em Lobachevskii J. Math.}, pages 47--70, 1999.

\bibitem[RF88]{RABTORI}
J.~Rabin and P.~Freund.
\newblock Supertori are algebraic curves.
\newblock {\em Comm. Math. Phys.}, 114:131--145, 1988.

\bibitem[RSV88]{ROSLYGSCONF}
A.~A. Rosly, A.~S. Schwarz, and A.~A. Voronov.
\newblock Geometry of superconformal manifolds.
\newblock {\em Comm. Math. Phys.}, (119):129--152, 1988.

\bibitem[Vai90]{VAIN}
Y.~Vaintrob.
\newblock Deformation of complex superspaces and coherent sheaves on them.
\newblock {\em J. Soviet Math.}, 51(1):2140--2188, August 1990.

\bibitem[Wit13]{WITTRS}
E.~Witten.
\newblock Notes on super {Riemann} surfaces and their moduli.
\newblock available at: \href{http://arxiv.org/abs/1209.2459}{arXiv:1209.2459}
  [hep-th], 2013.

\end{thebibliography}

\hfill
\\
\noindent
\textsc{
Kowshik Bettadapura
\\\\
Mathematical Sciences Institute, Australian National University, Canberra, ACT 2601, Australia
\\\\
Yau Mathematical Sciences Centre, Tsinghua University, Haidian, Beijing 100084, China}
\\\\
\emph{E-mail address:} \href{mailto:kowshik@mail.tsinghua.edu.cn}{kowshik@mail.tsinghua.edu.cn}

\end{document}